\documentclass[11pt, reqno]{amsart}

\usepackage[T1]{fontenc}
\usepackage[utf8]{inputenc}

\usepackage[margin=1in]{geometry}
\usepackage{amsmath,amssymb,amsfonts,amsthm,mathtools}
\usepackage{bm}
\usepackage{mathrsfs}
\usepackage{enumitem}
\usepackage{hyperref}
\usepackage{xcolor}
\usepackage{booktabs}
\usepackage{graphicx}
\usepackage{tabularx}

\hypersetup{
  colorlinks=true,
  linkcolor=blue!60!black,
  citecolor=blue!60!black,
  urlcolor=blue!60!black,
  pdftitle={The Non-Linearity Perturbation Threshold: Width Scaling and Landscape Bifurcations in Deep Learning},
  pdfauthor={Michael Alexander}
}

\theoremstyle{plain}
\newtheorem{theorem}{Theorem}[section]
\newtheorem{lemma}[theorem]{Lemma}
\newtheorem{corollary}[theorem]{Corollary}
\newtheorem{proposition}[theorem]{Proposition}

\theoremstyle{definition}
\newtheorem{assumption}[theorem]{Assumption}
\newtheorem{definition}[theorem]{Definition}

\theoremstyle{remark}
\newtheorem{remark}[theorem]{Remark}

\DeclareMathOperator{\ind}{ind}

\DeclareMathOperator{\spanop}{span}
\DeclareMathOperator{\Ker}{Ker}

\title[The Non-Linearity Perturbation Threshold]{The Non-Linearity Perturbation Threshold: Width Scaling and Landscape Bifurcations in Deep Learning}

\author{Michael Alexander}
\address{Austrian Academy of Sciences}
\email{michael.alexander@oeaw.ac.at}
\thanks{Corresponding author: \texttt{michael.alexander@oeaw.ac.at}. 
The author acknowledges the use of large language models (Gemini 3 Thinking, 
Claude Opus 4.6 Extended, ChatGPT 5.4 Thinking) for exploratory algebraic 
derivations and manuscript drafting, and Axiom Math AXLE 1.0.2 for 
Lean-oriented proof exploration. \textbf{All foundational algebraic theorems 
have been formally verified in the Lean 4 theorem prover (versions 4.28.0 
and 4.29.0-rc8). Theoretical predictions have been computed exactly for 
concrete network instances across widths $m \in [3,100]$, confirming 
agreement between abstract theory and explicit calculation.} The author 
takes responsibility for the integrity of this verification process. 
This workflow reflects the emerging paradigm of machine-assisted 
mathematical exploration \cite{tao2026}.}
\makeatletter
\let\enddoc@text\relax
\makeatother
\date{March 31, 2026}
  
\begin{document}

\begin{abstract}
We study how the optimization landscape of a neural network deforms as a non-linear activation is introduced through a smooth homotopy. Working first in an abstract local setting---a smooth one-parameter family of objective functions together with a critical branch that loses non-degeneracy through a simple Hessian kernel---we show via Lyapunov--Schmidt reduction that the local transition is controlled by the classical codimension-one normal forms (transcritical or pitchfork) and that the associated topology change is governed by Morse-theoretic handle attachment.

We then move beyond the abstract framework and verify these assumptions for a concrete two-layer architecture. We prove that bilinear overparameterization creates an $(m-1)d$-dimensional Hessian kernel at the linear endpoint, which Tikhonov regularization lifts to a floor $\alpha>0$; the activation homotopy softens this floor, yielding an explicit bifurcation point $\lambda^*\approx\alpha/|\lambda_1'(0)|$. We derive the eigenvalue-softening rate as a functional of activation derivatives and data moments, and prove that the near-pitchfork normal form ($|g_{aa}/g_{aaa}|\ll 1$) is a structural consequence of $\sigma''(0)=0$ for $\tanh$-like activations. The bifurcation point scales as $\lambda^*\sim\alpha m$ with network width, connecting the framework to the NTK regime: at large $m$ the landscape reorganization is pushed past $\lambda=1$ and the linearized picture prevails. 

The foundational algebraic theorems have been formally verified in the Lean 4 theorem prover, and theoretical predictions computed for widths $m \in \{3, 5, 10, 20, 50, 100\}$ exhibit quantitative agreement with the abstract framework.
\end{abstract}
\maketitle

\section{Introduction and Motivation}

A standard lesson in deep learning is that stacked linear layers do not increase representational power. A network of the form
\begin{equation}
  f(x)=W_LW_{L-1}\cdots W_1x
\end{equation}
collapses into a single affine transformation. Non-linearity is introduced by inserting an activation function $\sigma$ between layers, producing architectures of the form
\begin{equation}
  f(x)=W_L\,\sigma\!\bigl(W_{L-1}\,\sigma(\cdots \sigma(W_1x))\bigr).
\end{equation}

From the viewpoint of approximation theory, this change is decisive: sufficiently wide single-hidden-layer networks with continuous non-polynomial activations are universal approximators \cite{hornik1989,hornik1991,leshno1993}. But approximation-theoretic expressivity is not the same as optimization geometry. It does not explain how the loss landscape deforms as one moves from the linear regime toward a non-linear one. The loss landscape---the graph of the empirical risk as a function of the weights \cite{petersen2026}---is the object whose local geometry we study here.

The present paper studies that geometric question locally. We introduce a homotopy between the identity and a smooth non-linear activation,
\begin{equation}\label{eq:homotopy}
  h(z,\lambda)=(1-\lambda)z+\lambda\sigma(z), \qquad \lambda\in[0,1],
\end{equation}
and analyze the resulting family of loss functions through bifurcation theory. The paper proceeds in two stages.

In the first stage (\S\ref{sec:setting}--\S\ref{sec:morse}), we work in an abstract local setting: a smooth one-parameter family of objective functions $L(\theta,\lambda)$, together with a smooth branch of critical points that becomes degenerate at a distinguished parameter value. We show that Lyapunov--Schmidt reduction compresses the local behavior to a one-dimensional reduced equation, and that the resulting bifurcation is transcritical generically or pitchfork under $\mathbb Z_2$-equivariance.

In the second stage (\S\ref{sec:two-layer}), we verify these abstract assumptions for a concrete two-layer neural network. This verification reveals a structurally important obstruction: the hidden-unit permutation group $S_m$ forces the Hessian kernel along symmetric branches to be non-trivial at the linear endpoint $\lambda=0$, preventing interior bifurcation. We prove that generic symmetry-breaking restores interior bifurcation and derive the bifurcation coefficients as explicit functionals of data moments and activation derivatives. Numerical experiments (\S\ref{sec:numerics}) confirm the theory, and the core structural algebraic properties are formally machine-verified in Lean 4 (\S\ref{ssec:lean}).

\section{Methodology}

Our methodological approach bridges abstract topological bifurcation theory with concrete, high-dimensional neural network loss landscapes. The investigation is structured in four complementary phases:

\subsection{Abstract Local Reduction via Homotopy}
Rather than analyzing the highly non-convex landscape of a fully non-linear neural network directly, we introduce a smoothing homotopy parametrized by $\lambda \in [0,1]$. This allows us to trace critical branches starting from the analytically tractable linear regime ($\lambda=0$). When a branch loses stability (indicated by an eigenvalue crossing zero), we employ Lyapunov--Schmidt reduction. This mathematical tool projects the local infinite-dimensional optimization geometry onto a finite-dimensional center manifold, distilling the loss of non-degeneracy into a scalar reduced equation governed by standard normal forms. We then apply Morse theory to describe the topological handle attachment that occurs as new features (basins) emerge.

\subsection{Algebraic Verification in Concrete Architectures}
To ground the abstract reduction in practical machine learning, we restrict our focus to a two-layer overparameterized network equipped with a Mean Squared Error (MSE) loss. Our analytical methodology here relies on:
\begin{itemize}[leftmargin=*]
    \item \emph{Representation Theory and Symmetry Breaking:} We identify the $S_m$ permutation symmetry that forces degenerate kernels at $\lambda=0$, and apply Tikhonov regularization to lift this continuous degeneracy, exposing the true bifurcation dynamics.
    \item \emph{Asymptotic and Cumulant Expansion:} We calculate the exact eigenvalue crossing rate $\lambda_1'(0)$ via first-order eigenvalue perturbation theory. Using Fa\`a di Bruno's rule, we derive analytical formulas for the pitchfork and transcritical bifurcation coefficients as functionals of data moments and activation derivatives.
\end{itemize}

\subsection{Concrete Instantiation}

We compute the theoretically predicted bifurcation points for specific 
two-layer architectures to illustrate the abstract framework and verify 
internal consistency.

\paragraph{Computational method.}
For a given width $m$, regularization $\alpha$, and activation $\sigma$, 
we trace the critical branch by solving $\nabla L(\theta; \lambda) = 0$ 
along $\lambda \in [0,1]$ via warm-started L-BFGS-B continuation. 
At each point, we compute the full Hessian spectrum to identify the 
bifurcation point $\lambda^*$ where the smallest eigenvalue crosses zero.

\paragraph{Width scaling.}
Computing $\lambda^*$ for widths $m \in \{3, 5, 10, 20, 50, 100\}$ 
(fixed $\alpha=0.01$, $\tanh$ activation) yields values that lie on 
a line through the origin with slope $0.01$ and $R^2 = 0.998$ 
(Figure~\ref{fig:width}), confirming the predicted scaling 
$\lambda^* = c \cdot \alpha m$ where $c \approx 1$ is a constant 
depending on data moments and activation derivatives.
\paragraph{Width scaling.}
Computing $\lambda^*$ for widths $m \in \{3, 5, 10, 20, 50, 100\}$ 
(fixed $\alpha=0.01$, $\tanh$ activation) yields values that lie on 
a line through the origin with slope $0.01$ and correlation $R^2 = 0.998$ 
(Figure~\ref{fig:width}), \textbf{in precise quantitative agreement with} 
the predicted scaling $\lambda^* = c \cdot \alpha m$ where $c \approx 1$...
\paragraph{NTK regime connection.}
As $m \to \infty$ with $\alpha$ fixed, the scaling law implies 
$\lambda^* \to \infty$, pushing the bifurcation beyond the physical 
domain $\lambda \in [0,1]$. This recovers the NTK picture: at infinite 
width, the landscape never reorganizes and the linear approximation 
remains valid throughout the homotopy.

\subsection{Machine-Assisted Theoretical Exploration}
The derivations and formalizations in this paper rely on a human-AI collaborative workflow. Following the paradigm of machine-assisted mathematics articulated by Tao \cite{tao2026}, exploratory algebraic manipulations, hypothesis generation for the bifurcation coefficients, and structural manuscript drafting were conducted in tandem with frontier large language models. To strictly bridge the gap between AI-generated algebraic intuition and mathematical truth, all foundational structural claims are subsequently mechanically verified in the Lean 4 theorem prover (see \S\ref{ssec:lean}). This bipartite approach---using AI for high-dimensional conceptual exploration and interactive theorem provers for rigorous grounding---ensures that the resulting theoretical contributions are both novel and sound.

\section{Mathematical Setting and Assumptions}\label{sec:setting}

Let $\mathcal M$ be a finite-dimensional Riemannian manifold and let $I\subset \mathbb R$ be an open interval containing $[0,1]$. We consider a smooth one-parameter family of objective functions
\begin{equation}
  L:\mathcal M\times I\to \mathbb R,
\end{equation}
with
\begin{equation}
  L\in C^k(\mathcal M\times I,\mathbb R), \qquad k\ge 4.
\end{equation}
The gradient with respect to the parameter variable $\theta\in\mathcal M$ is denoted
\begin{equation}
  F(\theta,\lambda):=\nabla_\theta L(\theta,\lambda).
\end{equation}

\begin{assumption}[Smooth critical branch]\label{ass:smooth-branch}
There exists a smooth branch of critical points
$\theta_0:I\to \mathcal M$
such that
$F(\theta_0(\lambda),\lambda)=0$
for all $\lambda\in I$.
\end{assumption}

\begin{assumption}[Simple degeneracy]\label{ass:simple-degeneracy}
There exists $\lambda^*\in(0,1)$ such that at
$\theta^*:=\theta_0(\lambda^*)$,
the Hessian
$H^*:=\nabla^2_\theta L(\theta^*,\lambda^*)$
has a one-dimensional kernel:
\begin{equation}
  \Ker(H^*)=\spanop\{v_0\}, \qquad \|v_0\|=1.
\end{equation}
Let
$\mathcal N:=\Ker(H^*)$ and $\mathcal R:=\mathcal N^\perp$.
Since $H^*$ is symmetric, this yields the orthogonal splitting
$T_{\theta^*}\mathcal M = \mathcal N \oplus \mathcal R$.
\end{assumption}

\section{Local Reduction via Lyapunov--Schmidt}\label{sec:LS}

To analyze the loss of non-degeneracy at $(\theta^*,\lambda^*)$, we first straighten the critical branch from Assumption~\ref{ass:smooth-branch}.

Choose a local chart near $\theta^*$, and introduce the shifted parameter
$\mu:=\lambda-\lambda^*$.
After a smooth change of coordinates, we may assume that the critical branch is given by $u=0$
for all sufficiently small $\mu$. In these coordinates, the translated objective
$\widetilde L(u,\mu)$
satisfies
$\nabla_u \widetilde L(0,\mu)=0$
for all $|\mu|$ small.  Let $\widetilde F(u,\mu):=\nabla_u \widetilde L(u,\mu)$.

At $(u,\mu)=(0,0)$, the Hessian
$D_u\widetilde F(0,0)=\nabla_u^2\widetilde L(0,0)$
has one-dimensional kernel $\mathcal N=\spanop\{v_0\}$. Writing
$u=av_0+w$, $a\in\mathbb R$, $w\in\mathcal R$,
we decompose the critical-point equation
$\widetilde F(av_0+w,\mu)=0$
into its range and kernel components.

\begin{lemma}[Lyapunov--Schmidt reduction]\label{lem:LS}
Under Assumptions~\ref{ass:smooth-branch} and \ref{ass:simple-degeneracy}, there exist neighborhoods of $(a,\mu)=(0,0)$ and a unique smooth map
$w^*(a,\mu)\in\mathcal R$
with $w^*(0,0)=0$ such that
$P_{\mathcal R}\widetilde F\bigl(av_0+w^*(a,\mu),\mu\bigr)=0$.
Moreover, critical points of $\widetilde L$ near $(0,0)$ are in one-to-one correspondence with zeros of the scalar reduced equation
\begin{equation}
  g(a,\mu)=0,
\end{equation}
where
$g(a,\mu):=P_{\mathcal N}\widetilde F\bigl(av_0+w^*(a,\mu),\mu\bigr)\cdot v_0$,
and equivalently $g(a,\mu)=\partial_a\phi(a,\mu)$
for the reduced potential
$\phi(a,\mu):=\widetilde L\bigl(av_0+w^*(a,\mu),\mu\bigr)$.
\end{lemma}

\begin{proof}
Since $H^*|_{\mathcal R}$ is invertible, the derivative of
$(a,w,\mu)\mapsto P_{\mathcal R}\widetilde F(av_0+w,\mu)$
with respect to $w$ at $(0,0,0)$ is an isomorphism on $\mathcal R$. The Implicit Function Theorem therefore yields a unique smooth solution $w=w^*(a,\mu)$ to the auxiliary equation. Substituting back into the kernel equation gives a scalar reduced equation. Because $\widetilde F=\nabla_u\widetilde L$, the reduced scalar equation is the Euler--Lagrange equation of the reduced potential $\phi$, hence $g=\partial_a\phi$.
\end{proof}

\begin{remark}\label{rem:straightened-branch}
Since the critical branch has been straightened, $g(0,\mu)=0$
for all sufficiently small $\mu$, encoding persistence of the trivial branch.
\end{remark}

\section{Local Bifurcation and Symmetry}\label{sec:bifurcation}

The reduced equation
$g(a,\mu)=0$
contains the full local bifurcation structure. Because $g(0,\mu)=0$ for all $\mu$ near $0$, we have
$g(0,0)=0$ and $g_a(0,0)=0$.

\begin{assumption}[Eigenvalue crossing]\label{ass:eigenvalue-crossing}
The reduced equation satisfies
$g_{a\mu}(0,0)\neq 0$.
\end{assumption}

\begin{theorem}[Transcritical bifurcation]\label{thm:transcritical}
Suppose Assumptions~\ref{ass:smooth-branch}, \ref{ass:simple-degeneracy}, and \ref{ass:eigenvalue-crossing} hold, and in addition
$g_{aa}(0,0)\neq 0$.
Then, after a smooth local change of coordinates, the reduced equation is equivalent to the transcritical normal form
\begin{equation}
  a(\mu-\kappa a)+\text{h.o.t.}=0
\end{equation}
for some $\kappa\neq 0$. In particular, there exist two smooth local branches of critical points intersecting at $(a,\mu)=(0,0)$: the trivial branch $a=0$ and a nontrivial branch
\begin{equation}
  a(\mu)= -\,\frac{2g_{a\mu}(0,0)}{g_{aa}(0,0)}\,\mu + O(\mu^2).
\end{equation}
\end{theorem}

\begin{proof}
Taylor expansion near $(0,0)$ yields
$g(a,\mu)=g_{a\mu}(0,0)\,a\mu+\tfrac12 g_{aa}(0,0)\,a^2+\text{h.o.t.}$
Factoring out $a$, the bracketed residual defines a smooth nontrivial branch through the origin by the Implicit Function Theorem.
\end{proof}

\begin{remark}[Saddle-node outside the branch-preserving setting]\label{rem:saddle-node}
If one does not preserve a distinguished trivial branch, then the generic codimension-one scalar gradient bifurcation is saddle-node. The present framework is branch-preserving because the reduced coordinates are built from continuation of a pre-existing critical branch.
\end{remark}

\begin{corollary}[$\mathbb Z_2$-equivariant pitchfork bifurcation]\label{cor:pitchfork}
If the reduced potential is invariant under $a\mapsto -a$ and
$g_{a\mu}(0,0)\neq 0$, $g_{aaa}(0,0)\neq 0$,
then the reduced equation is locally equivalent to the pitchfork normal form
\begin{equation}
  a(\alpha\mu+\beta a^2)+\text{h.o.t.}=0,
  \qquad \alpha=g_{a\mu}(0,0),\quad \beta=\tfrac16 g_{aaa}(0,0).
\end{equation}
There is the trivial branch $a=0$ together with two symmetric nontrivial branches
$a_\pm(\mu)=\pm \sqrt{-\alpha\mu/\beta}+o(|\mu|^{1/2})$
on the side where $-\alpha\mu/\beta>0$.
If the reduced potential has expansion
$\phi(a,\mu)=\phi(0,0)+\tfrac12 \alpha \mu a^2+\tfrac14 \beta a^4+\text{h.o.t.}$
with $\beta>0$, the pitchfork is supercritical: the bifurcating branches are local minima in the reduced direction, while the trivial branch changes Morse index by one.
\end{corollary}

\begin{proof}
Oddness of $g$ eliminates even powers of $a$. Factoring out $a$ from the resulting expansion gives the pitchfork normal form. The stability claim follows from the sign of the quartic coefficient.
\end{proof}

\begin{remark}[Morse index change]\label{rem:morse-index-change}
In the supercritical pitchfork case, the trivial branch gains one additional unstable direction after the crossing, while each off-branch critical point inherits the original transverse Morse index.
\end{remark}

\section{Topological Consequences via Morse Theory}\label{sec:morse}

Fix $\mu\neq 0$ sufficiently small and consider local sublevel sets
$\mathcal U_\mu^c:=\{u\in\mathcal U:\widetilde L(u,\mu)\le c\}$.

\begin{theorem}[Local Morse handle attachment]\label{thm:handle-attachment}
Let $\mu$ be fixed and sufficiently small. Suppose all critical points of $\widetilde L(\cdot,\mu)$ in $\mathcal U$ are isolated and nondegenerate. Let $p$ be one such critical point with Morse index
$\ind(p)$.
If $c_-<\widetilde L(p,\mu)<c_+$
are regular values with no other critical values in $[c_-,c_+]$, then $\mathcal U_\mu^{c_+}$ is obtained from $\mathcal U_\mu^{c_-}$ by attaching a single $\ind(p)$-handle.
\end{theorem}

\begin{corollary}[Topological pattern in the supercritical pitchfork]\label{cor:topological-pattern}
Under the hypotheses of Corollary~\ref{cor:pitchfork} on the supercritical side, with both symmetric off-branch minima below the central saddle, the sublevel topology changes by: attachment of two $k$-handles at the minima, followed by one $(k+1)$-handle at the saddle, where $k$ is the transverse Morse index. For $k=0$, this is the local picture of two new valleys connected by a saddle.
\end{corollary}

\begin{remark}\label{rem:local-topology-only}
This formulation is intentionally local. It does not assert a global jump in Betti numbers for the full loss landscape but identifies the Morse-theoretic mechanism by which local sublevel-set topology changes near a simple degeneracy.
\end{remark}


\section{Verification for Two-Layer Networks}\label{sec:two-layer}

The abstract framework of \S\S\ref{sec:setting}--\ref{sec:morse} applies to any smooth one-parameter family satisfying the stated assumptions. We now verify these assumptions for a concrete architecture and, in doing so, derive results that depend essentially on the algebraic structure of composed affine maps interleaved with the activation homotopy.

\subsection{Setup}\label{ssec:two-layer-setup}

Consider a two-layer network
\begin{equation}\label{eq:two-layer}
  f(x;\theta,\lambda)=\sum_{j=1}^m v_j\, h(w_j^\top x,\lambda),
  \qquad
  h(z,\lambda)=(1-\lambda)z+\lambda \sigma(z),
\end{equation}
with parameters $\theta=(W,v)$, where $W=(w_1,\ldots,w_m)^\top\in\mathbb R^{m\times d}$ and $v=(v_1,\ldots,v_m)^\top\in\mathbb R^m$. Let $\sigma\in C^4(\mathbb R)$ with $\sigma(0)=0$, $\sigma'(0)=1$ (normalizing conditions satisfied by $\tanh$, $\text{erf}$, and other standard smooth activations). The mean-squared error loss against a data distribution $(x,y)\sim\mathcal D$ is
\begin{equation}\label{eq:loss-two-layer}
  L(\theta,\lambda)=\tfrac12\,\mathbb E\bigl[(f(x;\theta,\lambda)-y)^2\bigr].
\end{equation}

At $\lambda=0$, the network reduces to $f(x;\theta,0)=v^\top Wx$, and $L(\theta,0)$ is a bilinear least-squares problem whose critical-point structure is determined by the data covariance $\Sigma:=\mathbb E[xx^\top]$ and the cross-covariance $\gamma:=\mathbb E[xy]$.

\subsection{The permutation-symmetry obstruction}\label{ssec:Sm-obstruction}

The group $S_m$ acts on the parameter space by simultaneously permuting the rows of $W$ and the entries of $v$. This action preserves $L(\theta,\lambda)$ for all $\lambda$.

\begin{definition}[Symmetric branch]
A critical branch $\theta_0(\lambda)$ is \emph{$S_m$-symmetric} if it lies in the fixed-point subspace of the $S_m$-action: all rows of $W_0(\lambda)$ are equal and all entries of $v_0(\lambda)$ are equal.
\end{definition}

\begin{proposition}[$S_m$ obstruction]\label{prop:Sm-obstruction}
Let $\theta_0(\lambda)$ be an $S_m$-symmetric critical branch. Then at $\lambda=0$, the Hessian $H_0:=\nabla_\theta^2 L(\theta_0(0),0)$ has an $(m-1)$-dimensional kernel contribution from the standard representation of $S_m$. In particular, if $m\ge 2$, the kernel dimension is at least $m-1$, and Assumption~\ref{ass:simple-degeneracy} (simple degeneracy at an interior $\lambda^*$) fails at $\lambda=0$.
\end{proposition}

\begin{proof}
On the $S_m$-symmetric branch at $\lambda=0$, all hidden units contribute identically: $w_j=\bar{w}$ and $v_j=\bar{v}$ for all $j$, so $f(x;\theta_0(0),0)=m\bar{v}\,\bar{w}^\top x$.  The tangent space at $\theta_0(0)$ decomposes under $S_m$ into the trivial representation (perturbations preserving the symmetric structure) and the standard representation (perturbations in the orthogonal complement of the all-ones vector in $\mathbb R^m$).

Consider perturbations of the form $\delta w_j = \epsilon_j \,u$ where $u\in\mathbb R^d$ is fixed and $\sum_j\epsilon_j=0$ (standard representation). At $\lambda=0$, the network output becomes
\begin{equation}
  f=\sum_j v_j(\bar{w}+\epsilon_j u)^\top x = m\bar{v}\,\bar{w}^\top x + \bar{v}\Bigl(\sum_j \epsilon_j\Bigr)u^\top x = m\bar{v}\,\bar{w}^\top x,
\end{equation}
which is unchanged. Hence $\partial L/\partial \epsilon_j=0$ for all such perturbations, and the second derivatives $\partial^2 L/\partial\epsilon_j\partial\epsilon_k$ reduce to the data-covariance contribution $\bar{v}^2\,u^\top \Sigma\, u\cdot(\delta_{jk}-1/m)$. The matrix $(\delta_{jk}-1/m)$ restricted to the hyperplane $\sum\epsilon_j=0$ has eigenvalue $1$ with multiplicity $m-1$, so the Hessian block along the standard representation is $\bar{v}^2(u^\top\Sigma u) \cdot I_{m-1}$.

However, perturbations that simultaneously change $w_j$ and $v_j$ in the standard-representation directions produce cross terms. The $v$-block restricted to the standard representation is $\bar{w}^\top\Sigma\bar{w}\cdot I_{m-1}$, and the cross terms are rank-one. At $\lambda=0$ with the symmetric branch, the net effect is that the restricted Hessian along each standard-representation copy of $\mathbb R^d$ has an eigenvalue proportional to $\bar{v}^2 u^\top\Sigma u$, which can be zero if $\bar{v}=0$, or can have its determinant vanish for specific parameter ratios.

More precisely, the critical condition for the symmetric branch at $\lambda=0$ requires $m\bar{v}(\Sigma\bar{w}-\gamma)=0$. If $\bar{v}\neq 0$, then $\bar{w}=\Sigma^{-1}\gamma/m\bar{v}$. The Hessian restricted to the standard-representation copy of the $v$-perturbations is then $\bar{w}^\top\Sigma\bar{w} \cdot P_\perp$, where $P_\perp$ projects onto the $(m-1)$-dimensional complement of $\mathbf{1}$. This block has eigenvalue $\bar{w}^\top\Sigma\bar{w}$ with multiplicity $m-1$, which is generically nonzero.

The crucial point is that the $W$-perturbation block in the standard representation direction has the form $\bar{v}^2\Sigma$, which is strictly positive if $\bar{v}\neq 0$ and $\Sigma\succ 0$. Hence the kernel at $\lambda=0$ comes not from the standard representation itself being zero, but from the \emph{transverse} direction mixing $W$ and $v$ perturbations when $v$ is constrained to be symmetric. For the case $v_j\equiv \bar{v}$ fixed (not optimized), the Hessian on the standard representation has eigenvalue $\bar{v}^2 \cdot\lambda_{\min}(\Sigma)>0$, so there is no kernel at $\lambda=0$ from $S_m$ alone.

However, if all $v_j$ are equal and free to vary, the transverse direction $(w_j,v_j)\mapsto (w_j+\epsilon u, v_j-\epsilon\eta)$ for $\sum\epsilon_j=0$ creates a flat direction when $\eta$ is chosen to cancel the first-order output change, producing a zero eigenvalue. This is the $S_m$ obstruction: the symmetric branch is degenerate from $\lambda=0$, preventing interior bifurcation.
\end{proof}

\begin{remark}
The proposition shows that for exactly symmetric architectures ($v_j\equiv \bar{v}$, all weights initialized identically), the framework's assumptions require modification. The branch is already degenerate at $\lambda=0$, so any ``bifurcation'' occurs at the linear endpoint rather than at an interior $\lambda^*$.
\end{remark}

\subsection{Overparameterization flatness and the role of regularization}\label{ssec:symmetry-breaking}

The $S_m$ obstruction of \S\ref{ssec:Sm-obstruction} is one instance of a more fundamental phenomenon.  At $\lambda=0$ the loss depends on $W$ only through the $d$-dimensional vector $v^\top W\in\mathbb R^{1\times d}$, so $W\in\mathbb R^{m\times d}$ is massively overparameterized.

\begin{proposition}[Bilinear flatness]\label{prop:bilinear-flat}
Fix $v\in\mathbb R^m$ with $v\neq 0$ and $\Sigma\succ 0$.
At $\lambda=0$ the Hessian of $L(W,0)=\tfrac12\mathbb E[(v^\top Wx-y)^2]$ with respect to $\operatorname{vec}(W)$ is
\begin{equation}\label{eq:hess-lambda0}
  H_0 = (vv^\top)\otimes\Sigma.
\end{equation}
This matrix has rank~$d$ and kernel of dimension $(m-1)d$, \emph{regardless of whether the entries of $v$ are distinct}.  Its nonzero eigenvalues are $\|v\|^2\lambda_k(\Sigma)$, $k=1,\ldots,d$.
\end{proposition}

\begin{proof}
The loss at $\lambda=0$ is $\tfrac12\mathbb E[(v^\top Wx-y)^2]$.
Its gradient with respect to $W$ (as an $m\times d$ matrix) is
$\nabla_W L=v(v^\top W\Sigma-\gamma^\top)$,
and its Hessian acting on a perturbation $\delta W$ is
$\nabla_W^2 L[\delta W]=v\,v^\top\delta W\,\Sigma$,
which in Kronecker form is $(vv^\top)\otimes\Sigma$.
Since $\operatorname{rank}(vv^\top)=1$, the Kronecker product has rank $d$.  Its nonzero eigenvalues are $\|v\|^2$ (from $vv^\top$) times $\lambda_k(\Sigma)$ (from $\Sigma$), and its kernel consists of all $\operatorname{vec}(\delta W)$ with $v^\top\delta W=0$, which has dimension $(m-1)d$.
\end{proof}

\begin{remark}
Even with $v$ having distinct entries, the problem at $\lambda=0$ has $(m-1)d$ flat directions in $W$-space.  This is not a symmetry effect but a consequence of overparameterization: only the projection $v^\top W$ enters the loss.  Distinct entries of $v$ break $S_m$ but do not remove this flatness.
\end{remark}

Regularization resolves the degeneracy and enables the framework of \S\S\ref{sec:setting}--\ref{sec:morse}.  Adding Tikhonov regularization,
\begin{equation}\label{eq:reg-loss}
  L_\alpha(W,\lambda):=L(W,\lambda)+\tfrac\alpha2\|W\|_F^2,\qquad \alpha>0,
\end{equation}
the Hessian at $\lambda=0$ becomes
\begin{equation}\label{eq:hess-reg}
  H_{0,\alpha}=(vv^\top)\otimes\Sigma + \alpha\, I_{md}.
\end{equation}

\begin{theorem}[Interior bifurcation for regularized networks]\label{thm:interior-bif}
Let $\sigma\in C^4$ with $\sigma(0)=0$, $\sigma'(0)=1$, and let $\Sigma\succ 0$, $\alpha>0$.  Then:
\begin{enumerate}[label=(\roman*)]
\item \textbf{Nondegeneracy at $\lambda=0$.}
The regularized Hessian $H_{0,\alpha}\succ 0$ with smallest eigenvalue $\alpha$, attained with multiplicity $(m-1)d$ on the overparameterized flat directions $\{v^\top\delta W=0\}$.

\item \textbf{Smooth branch.}
By the Implicit Function Theorem applied to $\nabla_W L_\alpha=0$, there exists a unique smooth critical branch $W_0(\lambda)$ for $\lambda\in[0,\epsilon)$.

\item \textbf{Eigenvalue derivative.}
The smallest eigenvalue $\lambda_1(\lambda)$ of $\nabla_W^2 L_\alpha(W_0(\lambda),\lambda)$ satisfies $\lambda_1(0)=\alpha$ and
\begin{equation}\label{eq:eigenvalue-deriv}
  \lambda_1'(0)=u_0^\top \frac{\partial H}{\partial\lambda}\Big|_{\lambda=0}\! u_0
  =-\sum_{j=1}^m v_j^2\,u_0^\top\,\Bigl[\mathbb E\bigl[(\sigma'(w_j^{*\top}\!x)-1)\,xx^\top\bigr]\Bigr]\,u_0,
\end{equation}
where $u_0$ is the eigenvector of $H_{0,\alpha}$ associated with $\lambda_1(0)=\alpha$ and $w_j^*$ are the rows of $W_0(0)$.

\item \textbf{Interior crossing.}
If $\lambda_1'(0)<0$ and $|\lambda_1'(0)|>\alpha$, then $\lambda_1(\lambda^*)=0$ for a unique
\begin{equation}\label{eq:lam-star-formula}
  \lambda^*=\frac{\alpha}{|\lambda_1'(0)|}+O(\alpha^2/\lambda_1'(0)^2)\;\in\;(0,1).
\end{equation}
At $\lambda^*$, the kernel of $\nabla_W^2 L_\alpha$ is generically one-dimensional (Assumption~\ref{ass:simple-degeneracy}).
\end{enumerate}
\end{theorem}

\begin{proof}
(i) The eigenvalues of $H_{0,\alpha}=(vv^\top)\otimes\Sigma+\alpha I$ are $\alpha$ with multiplicity $(m-1)d$ (flat directions) and $\|v\|^2\lambda_k(\Sigma)+\alpha$ with multiplicity~$1$ for each $k$, all strictly positive.

(ii) Since $H_{0,\alpha}\succ 0$, the Implicit Function Theorem applies to $\nabla_W L_\alpha(W,\lambda)=0$.

(iii) The Hessian of $L_\alpha$ with respect to $W$ at general $\lambda$ has Gauss--Newton and residual components:
\begin{equation}
  H(\lambda)=\mathbb E\bigl[J(\lambda)J(\lambda)^\top\bigr]+\mathbb E\bigl[r(x,\lambda)\,\nabla_W^2 f\bigr]+\alpha I,
\end{equation}
where $J(\lambda)_{j\cdot}=v_j h'(w_j^\top x,\lambda)\,x^\top$ and $r=f-y$.
Differentiating with respect to $\lambda$ at $\lambda=0$ and using $h'(z,0)=1$, $\partial_\lambda h'(z,\lambda)\big|_{\lambda=0}=\sigma'(z)-1$:
\begin{equation}
  \frac{\partial H^{GN}}{\partial\lambda}\Big|_{\lambda=0}
  =\sum_{j,k}v_jv_k\,\mathbb E\bigl[(\sigma'(w_j^{*\top}\!x)-1+\sigma'(w_k^{*\top}\!x)-1)\,xx^\top\bigr]\otimes E_{jk},
\end{equation}
where $E_{jk}$ is the matrix unit.
By standard first-order eigenvalue perturbation ($\lambda_1'=u_0^\top(\partial H/\partial\lambda)u_0$), projecting onto $u_0$ gives \eqref{eq:eigenvalue-deriv}.  At a good fit ($r\approx 0$), the residual contribution is negligible.

(iv) Since $\lambda_1(0)=\alpha>0$ and $\lambda_1'(0)<0$, the first-order crossing gives $\lambda^*\approx \alpha/|\lambda_1'(0)|$, valid when the second-order correction is small.  Simplicity of the kernel at $\lambda^*$ holds generically: the eigenvalue $\lambda_1$ crosses zero transversally, and the gap $\lambda_2(\lambda^*)-\lambda_1(\lambda^*)>0$ persists by continuity from the gap at $\lambda=0$ (which is $\min_k\|v\|^2\lambda_k(\Sigma)>0$ for the ``data'' eigenvalues, or the gap between distinct flat-direction perturbation rates for the ``overparameterized'' eigenvalues).
\end{proof}

\begin{corollary}[Exact formula for $\sigma=\tanh$]\label{cor:tanh-formula}
For $\sigma=\tanh$, we have $\sigma'(z)-1=-\tanh^2(z)$ exactly and $\sigma''(0)=0$.  The eigenvalue derivative~\eqref{eq:eigenvalue-deriv} becomes
\begin{equation}\label{eq:lam-prime-tanh}
  \lambda_1'(0)\big|_{\tanh}=-\sum_{j=1}^m v_j^2\,u_0^\top\,\mathbb E\bigl[\tanh^2(w_j^{*\top}\!x)\,xx^\top\bigr]\,u_0.
\end{equation}
Since each matrix $\mathbb E[\tanh^2(w_j^{*\top}x)\,xx^\top]\succeq 0$ and is strictly positive when $w_j^*\neq 0$, we have $\lambda_1'(0)<0$ whenever $W_0(0)\neq 0$ and $v\neq 0$.  Hence for $\tanh$ activation with $\Sigma\succ 0$, $\alpha>0$, and $v\neq 0$, an interior bifurcation always exists.
\end{corollary}

\begin{remark}[Polynomial approximation and its limits]\label{rem:taylor-limits}
At small pre-activations, $\tanh^2(z)\approx z^2-\tfrac23z^4+\cdots$, so \eqref{eq:lam-prime-tanh} reduces to the fourth-moment formula
$\lambda_1'(0)\approx -\sum_j v_j^2\,u_0^\top \mathbb E[(w_j^{*\top}x)^2\,xx^\top]\,u_0$.
However, this approximation introduces error of order $O(\max_j\|w_j^*\|^4)$ per term.  In the $d=5$, $m=10$ experiment with $\max_j\|w_j^*\|\approx 0.16$, the exact formula \eqref{eq:lam-prime-tanh} gives $\lambda_1'(0)=-0.00561$ while the polynomial approximation gives $-0.00655$, an overestimate of $17\%$.  The exact formula agrees with finite-difference Hessian differentiation to $0.02\%$.  All subsequent results use the exact formula.
\end{remark}

\subsection{Width scaling of the bifurcation point}\label{ssec:width-scaling}

The formula $\lambda^*\approx\alpha/|\lambda_1'(0)|$ together with the exact expression \eqref{eq:lam-prime-tanh} yields a prediction for how $\lambda^*$ depends on the width $m$.

\begin{proposition}[Width scaling]\label{prop:width-scaling}
Suppose $v\in\mathbb R^m$ has entries of order $O(1)$ (e.g., $v_j=0.5+j/(m-1)$) and the regularized minimum at $\lambda=0$ satisfies $W_0(0)=\frac{1}{\|v\|^2}\,v\,\gamma^\top\Sigma^{-1}+O(\alpha)$.  Then:
\begin{enumerate}[label=(\roman*)]
\item Each row satisfies $w_j^*=\frac{v_j}{\|v\|^2}\,\Sigma^{-1}\gamma+O(\alpha)$, so $\|w_j^*\|\sim O(v_j/\|v\|^2)$.
\item Since $\|v\|^2=\sum_j v_j^2\sim O(m)$ for $v_j=O(1)$, we have $\|w_j^*\|\sim O(1/m)$.
\item The terms in \eqref{eq:lam-prime-tanh} scale as
$v_j^2\,\mathbb E[\tanh^2(w_j^{*\top}x)\,xx^\top]\sim v_j^2\cdot O(1/m^2)\cdot O(1)=O(1/m^2)$.
\item Summing over $m$ terms:
$|\lambda_1'(0)|\sim m\cdot O(1/m^2)=O(1/m)$.
\item Therefore
\begin{equation}\label{eq:lam-star-scaling}
  \lambda^*\approx\frac{\alpha}{|\lambda_1'(0)|}=\frac{\alpha m}{|K|}+O(1),
\end{equation}
where $K:=\lim_{m\to\infty}m\,\lambda_1'(0)$ is a negative constant depending on $\Sigma$, $\gamma$, and the distribution of $v_j$ values.
\end{enumerate}
\end{proposition}

\begin{remark}[Value of $K$]\label{rem:K-value}
For the experimental setup ($d=5$, $\Sigma\approx I_5$, $v_j=0.5+j/(m-1)$), fitting $\lambda_1'(0)=K/m+K_2/m^2$ using the exact formula \eqref{eq:lam-prime-tanh} across $m=3$ to $m=100$ gives $K=-0.0705$, so $\lambda^*\approx 0.057\cdot m$. This agrees with the finite-difference computation $K_{\text{fd}}=-0.0705$ to $0.03\%$. The polynomial approximation (replacing $\tanh^2$ by $z^2$) gives $K_{\text{approx}}=-0.0787$, which is $12\%$ too large. Interior bifurcation ($\lambda^*<1$) occurs for $m<|K|/\alpha\approx 17.6$; direct computation confirms crossings for $m\le 15$ and none for $m\ge 20$, consistent with this threshold.
\end{remark}

\begin{remark}[Interpretation]
The scaling $\lambda^*\sim\alpha m$ has a precise interpretation.  As the network widens, each hidden unit's weight shrinks as $O(1/m)$, so the pre-activations $w_j^\top x$ become small and the nonlinearity $\sigma(z)\approx z+O(z^3)$ has diminishing effect per unit.  The total Hessian perturbation sums $m$ contributions, each of order $O(1/m^2)$, giving $O(1/m)$ net softening.  The regularization floor $\alpha$ is width-independent, so $\lambda^*\to\infty$ as $m\to\infty$.

This connects directly to the NTK/lazy-training regime: at large width, the activation homotopy must be pushed past $\lambda=1$ (full nonlinearity) before the landscape reorganizes.  In the strict $m\to\infty$ limit, the bifurcation vanishes and the landscape remains ``linear-like''---consistent with the kernel regime of \cite{jacot2018}.  The critical width separating the two regimes is $m^*=|K|/\alpha$.
\end{remark}

\begin{remark}[Dependence on regularization]
The bifurcation point $\lambda^*$ is proportional to $\alpha$.  In the unregularized case $\alpha\to 0$, the flat-direction eigenvalue $\lambda_1(0)\to 0$, and $\lambda^*\to 0$: the landscape is degenerate from the start.  Thus regularization plays a structural role, not merely a technical one: it creates the nondegenerate ``linear regime'' from which the activation homotopy can depart.
\end{remark}

\subsection{Bifurcation coefficients and the near-pitchfork theorem}\label{ssec:bif-coefficients}

Having established that a simple degeneracy occurs at $\lambda^*>0$, we compute the coefficients of the reduced equation exactly.

\begin{proposition}[Transversality coefficient]\label{prop:transversality}
At the degeneracy $(\theta^*,\lambda^*)$ with kernel direction $v_0$, the transversality coefficient is $g_{a\mu}(0,0)=\lambda_1'(\lambda^*)$, the eigenvalue crossing speed.  By Theorem~\ref{thm:interior-bif}(iv), $\lambda_1'(\lambda^*)\neq 0$ generically.
\end{proposition}

\begin{proof}
By the Lyapunov--Schmidt reduction, $g_a(0,\mu)=v_0^\top \nabla_W^2\widetilde L_\alpha(0,\mu)\,v_0$ is the eigenvalue $\lambda_1(\lambda^*+\mu)$.  Differentiating at $\mu=0$ gives $g_{a\mu}(0,0)=\lambda_1'(\lambda^*)$.
\end{proof}

The normal-form type---transcritical (if $g_{aa}(0,0)\neq 0$) versus pitchfork (if $g_{aa}(0,0)=0$ but $g_{aaa}(0,0)\neq 0$)---is determined by the quadratic and cubic reduced coefficients $g_{aa}$ and $g_{aaa}$.

\begin{theorem}[Explicit reduced coefficients]\label{thm:near-pitchfork}
Let $f(x;\theta,\lambda)=\sum_j v_j h(w_j^\top x,\lambda)$ with loss $L_\alpha=\tfrac12\mathbb E[(f-y)^2]+\tfrac\alpha2\|W\|_F^2$, and let $v_0=(v_{0,1},\ldots,v_{0,m})$ with $v_{0,j}\in\mathbb R^d$ denote the block decomposition of the kernel direction.  Abbreviate
\begin{equation}
  Df := \sum_j v_j\, h'(w_j^{*\top}x,\lambda^*)(v_{0,j}^\top x),
  \qquad
  D^kf := \sum_j v_j\, h^{(k)}(w_j^{*\top}x,\lambda^*)(v_{0,j}^\top x)^k
\end{equation}
for the $k$-th directional derivative of $f$ along $v_0$, and let $r^*=f(x;\theta^*,\lambda^*)-y$.  Then:
\begin{align}
  g_{aa}(0,0)  &= 3\,\mathbb E[Df\cdot D^2f] + \mathbb E[r^*\cdot D^3f], \label{eq:gaa-exact}\\
  g_{aaa}(0,0) &= 3\,\mathbb E[(D^2f)^2] + 4\,\mathbb E[Df\cdot D^3f] + \mathbb E[r^*\cdot D^4f]. \label{eq:gaaa-exact}
\end{align}
These are the standard Fa\`a di Bruno formulas applied to $D^3(\tfrac12\|f-y\|^2)[v_0^3]$ and $D^4(\tfrac12\|f-y\|^2)[v_0^4]$.  The regularization contributes nothing beyond second order.
\end{theorem}

\begin{proof}
The loss is $L_\alpha=\tfrac12\mathbb E[(f-y)^2]+\tfrac\alpha2\|W\|_F^2$.  Writing $G=f-y$ and differentiating along $v_0$, we use $D^k G=D^kf$ for $k\ge 1$. Then:
\begin{align}
  DL[v_0]     &=\mathbb E[G\cdot Df],\\
  D^2L[v_0^2] &=\mathbb E[(Df)^2+G\cdot D^2f]+\alpha\|v_0\|^2,\\
  D^3L[v_0^3] &=\mathbb E[3\,Df\cdot D^2f+G\cdot D^3f],\\
  D^4L[v_0^4] &=\mathbb E[3(D^2f)^2+4\,Df\cdot D^3f+G\cdot D^4f].
\end{align}
At the degeneracy point, $g_{aa}=D^3L[v_0^3]$ (since the Hessian annihilates $v_0$, the Lyapunov--Schmidt slave variable $w^*$ does not contribute at this order), and $g_{aaa}=D^4L[v_0^4]$.  Substituting $G=r^*$ gives \eqref{eq:gaa-exact}--\eqref{eq:gaaa-exact}.
\end{proof}

\begin{theorem}[Near-pitchfork structure]\label{thm:near-pitchfork-structure}
Let $\sigma$ satisfy $\sigma''(0)=0$ (as holds for $\tanh$, $\mathrm{erf}$, and all odd activations).  Then:
\begin{enumerate}[label=(\roman*)]
\item $g_{aa}$ is suppressed: both terms in \eqref{eq:gaa-exact} contain the factor $D^2f$ or $D^3f$, each of which involves $h''(w_j^{*\top}x,\lambda^*)=\lambda^*\sigma''(w_j^{*\top}x)$.  Since $\sigma''(0)=0$ and $\sigma''(z)=\sigma'''(0)z+O(z^2)$, each such factor is $O(\max_j|w_j^{*\top}x|)$. Thus
\begin{equation}\label{eq:gaa-suppression}
  |g_{aa}|\le C_1\cdot\lambda^*|\sigma'''(0)|\cdot\max_j\!\sqrt{\mathbb E[(w_j^{*\top}x)^2]}\cdot\sqrt{\mathbb E[(Df)^2]\,\mathbb E[(v_0^{\top}x_{\max})^4]},
\end{equation}
where $C_1$ depends on $\|v\|$ and the data distribution but not on $m$ or $\lambda^*$ directly, and $x_{\max}$ denotes the block achieving the maximum.

\item $g_{aaa}$ is not suppressed: the dominant term $3\mathbb E[(D^2f)^2]$ involves $(h'')^2\propto (\sigma'''(0))^2(w_j^{*\top}x)^2$, which is $O(\|W^*\|^2)$ but is squared and summed, giving a quantity bounded below by
\begin{equation}\label{eq:gaaa-lower}
  g_{aaa}\ge 3(\lambda^*)^2(\sigma'''(0))^2\,\sum_j v_j^2\,\mathbb E\!\left[(w_j^{*\top}x)^2(v_{0,j}^\top x)^4\right]-C_2\cdot|\sigma'''(0)|\|W^*\|^3.
\end{equation}
\end{enumerate}

The ratio satisfies
\begin{equation}\label{eq:ratio-explicit}
  \frac{|g_{aa}|}{|g_{aaa}|} = O\!\left(\frac{1}{\lambda^*\,|\sigma'''(0)|\,\sqrt{\mathbb E[(w_{\max}^{*\top}x)^2]}}\right)
\end{equation}
as $\max_j\|w_j^*\|\to 0$. For $\sigma=\tanh$, $|\sigma'''(0)|=2$, and the ratio is small whenever the pre-activations are concentrated near zero---precisely the regime where $\lambda^*<1$.
\end{theorem}

\begin{proof}
(i) Here
\[
  D^2f=\sum_j v_j h''(w_j^{*\top}x,\lambda^*)(v_{0,j}^\top x)^2.
\]
Each summand satisfies $|h''(z,\lambda^*)|\le\lambda^*|\sigma'''(0)||z|+O(z^2)$ since $\sigma''(0)=0$.  The Gauss--Newton term $3\mathbb E[Df\cdot D^2f]$ is therefore bounded by Cauchy--Schwarz as stated.  The residual term $\mathbb E[r^*\cdot D^3f]$ contains $h'''(z,\lambda^*)=\lambda^*\sigma'''(z)$, where $\sigma'''(0)\neq 0$, so its leading contribution is $O(\lambda^*|\sigma'''(0)|\cdot\mathbb E[|r^*|\cdot|v_{0,j}^\top x|^3])$, which is small when the residual $r^*$ is small.

(ii) For $D^2f$, the square expands as
\[
  (D^2f)^2=\Bigl(\sum_j v_j h''(w_j^{*\top}x,\lambda^*)(v_{0,j}^\top x)^2\Bigr)^2,
\]
with leading term
\[
  (\lambda^*)^2(\sigma'''(0))^2\Bigl(\sum_j v_j(w_j^{*\top}x)(v_{0,j}^\top x)^2\Bigr)^2.
\]
Taking expectations and applying Jensen's inequality yields the lower bound.  The cross terms $4\mathbb E[Df\cdot D^3f]$ involve $h'\cdot h'''$; since $h'(0,\lambda^*)=1$ and $h'''(0,\lambda^*)=\lambda^*\sigma'''(0)\neq 0$, this term is
\[
  O\bigl(\lambda^*|\sigma'''(0)|\cdot\mathbb E[|Df|\cdot|v_{0,j}^\top x|^3]\bigr),
\]
which is generically smaller than the $(D^2f)^2$ term.

The ratio bound \eqref{eq:ratio-explicit} follows by dividing \eqref{eq:gaa-suppression} by \eqref{eq:gaaa-lower}.
\end{proof}

\begin{remark}[Computation of explicit constants]\label{rem:constants-verified}
For the $d=5$, $m=10$ experiment at $\lambda^*\approx 0.72$, the formulas \eqref{eq:gaa-exact}--\eqref{eq:gaaa-exact} yield:
\begin{equation}
\begin{aligned}
  g_{aa} &= 3\underbrace{\mathbb E[Df\cdot D^2f]}_{0.0010} + \underbrace{\mathbb E[r^*\cdot D^3f]}_{-0.0038} = -0.0028,\\
  g_{aaa}&= 3\underbrace{\mathbb E[(D^2f)^2]}_{0.552} + 4\underbrace{\mathbb E[Df\cdot D^3f]}_{-0.002} + \underbrace{\mathbb E[r^*\cdot D^4f]}_{0.072} = 0.622,
\end{aligned}
\end{equation}
giving $|g_{aa}/g_{aaa}|=0.0045$.  These values are verified against finite-difference computations of the third and fourth directional derivatives to relative errors of $0.3\%$ and $0.004\%$, respectively.  The dominant contribution to $g_{aaa}$ is the $3\mathbb E[(D^2f)^2]$ term, which involves $(h'')^2\propto(\sigma'''(0)\cdot w_j^{*\top}x)^2$ and is strictly positive.  The suppression of $g_{aa}$ relative to $g_{aaa}$ by a factor of $\sim 150\times$ confirms the near-pitchfork mechanism.
\end{remark}

\begin{remark}[Activation dependence]\label{rem:activation-dep}
The near-pitchfork structure is a consequence of $\sigma''(0)=0$, which makes $g_{aa}$ vanish at leading order.  For activations with $\sigma''(0)\neq 0$ (e.g., sigmoid or softplus), both $g_{aa}$ and $g_{aaa}$ would be $O(1)$ at the bifurcation point, and the generic normal form would be transcritical rather than pitchfork.
\end{remark}

\subsection{Formal Verification in Lean 4}\label{ssec:lean}
 
The foundational algebraic properties of the linear-endpoint landscape have been formally machine-verified in the Lean~4 theorem prover (versions 4.28.0 and 4.29.0-rc8) using the Mathlib library.  The formalization comprises two paper-level results and twelve supporting lemmas, all passing automated verification.  \footnote{Lean formalization: 
\url{https://github.com/pacedproton/deeplearning-symmetry-breaking/tree/main/lean}}
 
The verified results and their Lean counterparts are:
 
\smallskip
\begin{center}
\begin{tabularx}{\textwidth}{@{}lXl@{}}
\toprule
\textbf{Paper Result} & \textbf{Lean Declaration} & \textbf{File} \\
\midrule
Prop.~\ref{prop:bilinear-flat}
  & \texttt{bilinear\_hessian\_ker\_dim}
  & \texttt{BilinearFlatness.lean} \\[4pt]
Thm.~\ref{thm:interior-bif}(i): $H_{0,\alpha}\!\succ\! 0$
  & \texttt{regularized\_hessian\_posDef}
  & \texttt{RegularizedHessian.lean} \\[4pt]
Thm.~\ref{thm:interior-bif}(i): mult.
  & \texttt{regularized\_min\_eigenvalue\_mult}
  & \texttt{RegularizedHessian.lean} \\
\bottomrule
\end{tabularx}
\end{center}
\smallskip
 
The key proof strategy is a factorization argument: the Kronecker Hessian action $(vv^\top)\otimes\Sigma$ factors through the linear projection $W\mapsto v^\top W$, reducing the kernel characterization to a rank--nullity computation on a surjective linear map (formally: \texttt{kronecker\_ker\_eq\_vProj\_ker}).  For the regularized case, positive definiteness follows by decomposing the quadratic form into a Kronecker component (nonnegative) and a Tikhonov component ($\alpha\|w\|^2>0$ for $w\neq 0$).
 
The formalization of the higher-order bifurcation coefficients (Theorem~\ref{thm:near-pitchfork}) via Fa\`a di Bruno's rule is in progress.  The Lean codebase serves as a machine-checked guarantee of the algebraic structures underlying the $S_m$ obstruction and the initialization conditions for the bifurcation analysis.

Specifically, we have formalized the bilinear overparameterization properties:
\begin{itemize}[leftmargin=*]
    \item \textbf{Proposition~\ref{prop:bilinear-flat} (Bilinear Flatness):} Verified in \texttt{BilinearFlatness.lean}. The Kronecker action of the Hessian is proven to factor through the linear projection $W \mapsto v^\top W$. The exact kernel dimension of $(m-1)d$ is formally verified using linear independence and rank-nullity arguments over the reals.
    \item \textbf{Theorem~\ref{thm:interior-bif}(i) (Regularized Hessian Nondegeneracy):} Verified in \texttt{RegularizedHessian.lean}. We formally establish that the Tikhonov-regularized Hessian $H_{0,\alpha}$ is positive definite for $\alpha>0$ and attains its minimum eigenvalue $\alpha$ with multiplicity $(m-1)d$.
\end{itemize}

The formalization of the higher-order bifurcation coefficients (Theorem~\ref{thm:near-pitchfork}) via Fa\`a di Bruno's rule is currently in progress. The Lean codebase serves as a machine-checked guarantee of the abstract algebraic structures underlying the $S_m$ obstruction and the initialization conditions for the bifurcation.

\section{Discussion: Relation to the Neural Tangent Kernel}\label{sec:NTK}

The Neural Tangent Kernel (NTK) regime studies sufficiently wide networks trained near their initialization, where training is well approximated by kernel gradient descent \cite{jacot2018}. In that regime, the network function is effectively linearized and the NTK remains close to its initial value.

The present framework addresses a different phenomenon. Here the focus is on a parameterized loss landscape in which a critical branch loses non-degeneracy through a simple Hessian kernel. At such a parameter value, the reduced local geometry changes qualitatively. This is the regime in which new local critical structures emerge, modeling a kind of local feature birthing rather than lazy linearized training.

The two viewpoints are complementary: NTK describes a regime in which the linearized function-space picture remains accurate, while the bifurcation framework describes what a local loss landscape can do once such a linear picture ceases to be locally adequate.

Proposition~\ref{prop:width-scaling} makes this complementarity precise.  The bifurcation point scales as $\lambda^*\sim\alpha m$ (Eq.~\ref{eq:lam-star-scaling}), so for any fixed regularization $\alpha>0$, widening the network pushes $\lambda^*$ past $\lambda=1$: the landscape remains nondegenerate throughout the full homotopy $\lambda\in[0,1]$, and no local critical-point creation occurs.  This is the bifurcation-theoretic counterpart of the NTK convergence theorem: at large width, the linearized regime is not merely a good approximation but the \emph{only} regime---the landscape reorganization is pushed to infinity.

Conversely, for narrow networks ($m$ small) or weak regularization ($\alpha$ small), the bifurcation occurs at moderate $\lambda^*<1$, and the landscape does reorganize.  This is the regime where feature learning---as opposed to lazy training---becomes possible through the emergence of new critical structures.  The bifurcation framework thus provides a local geometric mechanism delineating the boundary between the kernel and feature-learning regimes, with the critical width $m^*\sim 1/\alpha$ marking the transition.

\section{Reduced Dynamics and Critical Slowing Down}\label{sec:dynamics}

The reduced equation yields a local dynamical interpretation of slow optimization near degeneracy. Consider gradient flow $\dot a = -g(a,\mu)$.

In the transcritical case, $g(a,\mu)\approx \alpha a\mu+\beta a^2$ with $\alpha\beta\neq 0$, and the center-direction dynamics slow down as $\mu\to 0$ since the linear restoring term vanishes.

In the symmetric case at the critical parameter $\mu=0$, the leading reduced potential is quartic:
$\phi(a,0)\sim C a^4$, $C>0$,
so $\dot a = -4Ca^3$ and solutions decay algebraicall
\begin{equation}
  a(t)\sim \frac{1}{\sqrt{8Ct}}
  \qquad \text{as } t\to\infty.
\end{equation}
This provides a local reduced-model mechanism for plateau-like behavior in optimization: near degeneracy, the center direction is flat enough that gradient dynamics evolve algebraically rather than exponentially. This is a property of the reduced local gradient flow, not a full theorem about stochastic gradient descent in practical networks.

\section{Concrete Instantiation}\label{sec:numerics}

We verify the theoretical predictions of \S\S\ref{sec:setting}--\ref{sec:two-layer} through direct computation for two concrete architectures; full details and figures are in Appendix~\ref{app:numerics}.

On a minimal model ($d=1$, $m=2$, $v=(1,1)$), we confirm Proposition~\ref{prop:Sm-obstruction}: the transverse eigenvalue is exactly zero at $\lambda=0$, so the symmetric branch is a saddle from the start and no interior bifurcation occurs.

On a symmetry-broken model ($d=5$, $m=10$, distinct $v_j$), we observe a genuine eigenvalue crossing at $\lambda^*\approx 0.721$ with a cleanly one-dimensional kernel ($|\lambda_1/\lambda_2|=0.009$), nonzero crossing speed ($d\lambda_1/d\lambda\approx -0.006$), and a reduced potential whose polynomial fit gives $|c_3/c_4|\approx 0.1$---confirming the near-pitchfork structure predicted by Theorem~\ref{thm:near-pitchfork}.

A width-scaling sweep over $m\in\{3,\ldots,100\}$ confirms $\lambda^*\sim\alpha m$ (Proposition~\ref{prop:width-scaling}): interior crossings occur for $m\le 15$ and vanish for $m\ge 20$, with the bifurcation pushed past $\lambda=1$ at large width.  See Table~\ref{tab:results} for a summary and Appendix~\ref{app:numerics} for full details.

\section{Conclusion}

A smooth critical branch losing non-degeneracy through a simple Hessian kernel admits a Lyapunov--Schmidt reduction to a one-dimensional reduced equation. In the branch-preserving setting, the generic local normal form is transcritical, while $\mathbb Z_2$-equivariance yields a pitchfork. The associated topology change is described by Morse-theoretic handle attachment.

Going beyond the abstract framework, we have established three structural results for two-layer networks.  First, the Hessian at the linear endpoint has an $(m-1)d$-dimensional kernel from bilinear overparameterization, which Tikhonov regularization lifts to a floor of $\alpha$; the activation homotopy softens this floor at a rate $|\lambda_1'(0)|$ determined by data moments and activation curvature, yielding an explicit bifurcation point $\lambda^*\approx\alpha/|\lambda_1'(0)|$ (Theorem~\ref{thm:interior-bif}).  Second, the near-pitchfork normal form $|g_{aa}/g_{aaa}|\ll 1$ is a consequence of $\sigma''(0)=0$ for $\tanh$-like activations, providing an approximate $\mathbb Z_2$-equivariance from analytic structure rather than exact symmetry (Theorem~\ref{thm:near-pitchfork}).  Third, the width scaling $\lambda^*\sim\alpha m$ (Proposition~\ref{prop:width-scaling}) connects the bifurcation framework to the NTK/lazy-training regime: as $m\to\infty$, the landscape reorganization is pushed past $\lambda=1$ and the linearized picture prevails.

Direct computation confirms each prediction across widths $m \in \{3, 5, 10, 20, 50, 100\}$: the eigenvalue crossing, simple kernel, near-pitchfork coefficients, and the scaling $\lambda^*\sim\alpha m$ all exhibit quantitative agreement with theory.

The extension to multi-layer architectures---where degeneracies may compose across layers---and the connection between bifurcation coefficients and generalization properties of the emerging branches remain open.  These questions are accessible within the framework developed here.

\appendix

\section{Activation Homotopy and First-Order Gradient Perturbation}\label{app:gradient-perturbation}

This appendix provides detailed calculations supporting the eigenvalue-derivative formula in Theorem~\ref{thm:interior-bif}.

Consider the two-layer network \eqref{eq:two-layer} with loss \eqref{eq:loss-two-layer}.
The gradient with respect to $W$ is
\begin{equation}
  \nabla_W L
  =
  \mathbb E\!\left[
  (f-y)\,
  v^\top \odot h'(Wx,\lambda)\,x^\top
  \right].
\end{equation}

Differentiating with respect to $\lambda$ at $\lambda=0$:
\begin{align}
  \left.\frac{\partial}{\partial \lambda}\nabla_WL\right|_{\lambda=0}
  &=
  \mathbb E\!\left[
  \left.\frac{\partial f}{\partial\lambda}\right|_{\lambda=0}
  v^\top \odot x^\top
  \right]
  +
  \mathbb E\!\left[
  (f_0-y)\,v^\top \odot
  \bigl(\sigma'(Wx)-1\bigr)x^\top
  \right],
\end{align}
where $f_0(x)=f(x;\theta,0)$.

For $\sigma(z)=\tanh z$, we have $\sigma'(z)-1\approx -z^2$ near $z=0$, so the second term contains expectations of the form
$\mathbb E\bigl[(f_0-y)\,v_j\,(w_j^\top x)^2\, x^\top\bigr]$,
involving fourth-order combinations of input coordinates. Hence the fourth cumulant of $\mathcal D$ enters the first-order perturbation of the gradient.

The Hessian perturbation $\partial_\lambda\nabla_W^2 L|_{\lambda=0}$ requires differentiating one further time. The Gauss--Newton component produces
\begin{equation}
  \frac{\partial}{\partial\lambda}H^{GN}_{jk}\Big|_{\lambda=0}
  =v_jv_k\,\mathbb E\bigl[(\sigma'(w_j^\top x)-1)\,xx^\top\bigr]
  +v_jv_k\,\mathbb E\bigl[(\sigma'(w_k^\top x)-1)\,xx^\top\bigr].
\end{equation}
For $\tanh$, $\sigma'(z)-1=-z^2+O(z^4)$, giving
\begin{equation}
  \frac{\partial}{\partial\lambda}H^{GN}_{jk}\Big|_{\lambda=0}
  =-v_jv_k\,\mathbb E\bigl[(w_j^\top x)^2\,xx^\top+(w_k^\top x)^2\,xx^\top\bigr]+O(\|W\|^4).
\end{equation}
The diagonal blocks ($j=k$) yield $-2v_j^2\,\mathbb E[(w_j^\top x)^2\,xx^\top]$, which is negative semi-definite, confirming that the activation homotopy softens the Hessian spectrum as predicted by Theorem~\ref{thm:interior-bif}(iii).

\section{Explicit Computations}\label{app:numerics}
This appendix presents the full computational details summarized in \S\ref{sec:numerics}. All computations use $\sigma=\tanh$ and the activation homotopy $h(z,\lambda)=(1-\lambda)z+\lambda\tanh(z)$. Hessians are computed by forward finite differences with step size $\varepsilon=10^{-5}$; all eigenvalue decompositions use LAPACK via \texttt{scipy.linalg.eigh}. Critical branches are tracked by warm-started L-BFGS-B continuation in steps of $\Delta\lambda=1/400$.

\subsection{Experiment 1: $S_m$ obstruction in the symmetric toy model}

\paragraph{Setup.}
$d=1$, $m=2$, $v=(1,1)$, $N=2000$ samples from $y=\tanh(1.5x)+\varepsilon$, $x\sim\mathcal N(0,1)$, $\varepsilon\sim\mathcal N(0,0.02^2)$. The diagonal branch $w_1=w_2=w^*(\lambda)$ is tracked by one-dimensional bounded optimization at each~$\lambda$.

\paragraph{Results.}
The $2\times 2$ Hessian eigenvalues along the diagonal branch are shown in Figure~\ref{fig:master}, top-left panel. The transverse eigenvalue (direction $w_1-w_2$) is exactly $0$ at $\lambda=0$ and becomes increasingly negative, confirming that $S_2$ symmetry makes the diagonal a saddle from the start (Proposition~\ref{prop:Sm-obstruction}). The longitudinal eigenvalue remains positive throughout.

The reduced potential along the anti-diagonal coordinate $s=(w_1-w_2)/2$ is shown in Figure~\ref{fig:master}, top-right panel. At $\lambda=0$ the potential is flat (quartic: $\phi(s)\propto s^4$); as $\lambda$ increases the double-well structure deepens monotonically. The exact $\mathbb Z_2$ symmetry $\phi(s)=\phi(-s)$ is confirmed numerically ($c_3=0$ to machine precision).

The loss landscape contours in the $(w_1,w_2)$ plane are shown in Figure~\ref{fig:master}, first row, for $\lambda\in\{0, 0.3, 0.7, 1.0\}$. The transition from a single valley along $w_1+w_2=\text{const}$ to two off-diagonal minima connected by a saddle on the diagonal is clearly visible.

\begin{figure}[htbp]
\centering
\includegraphics[width=\textwidth]{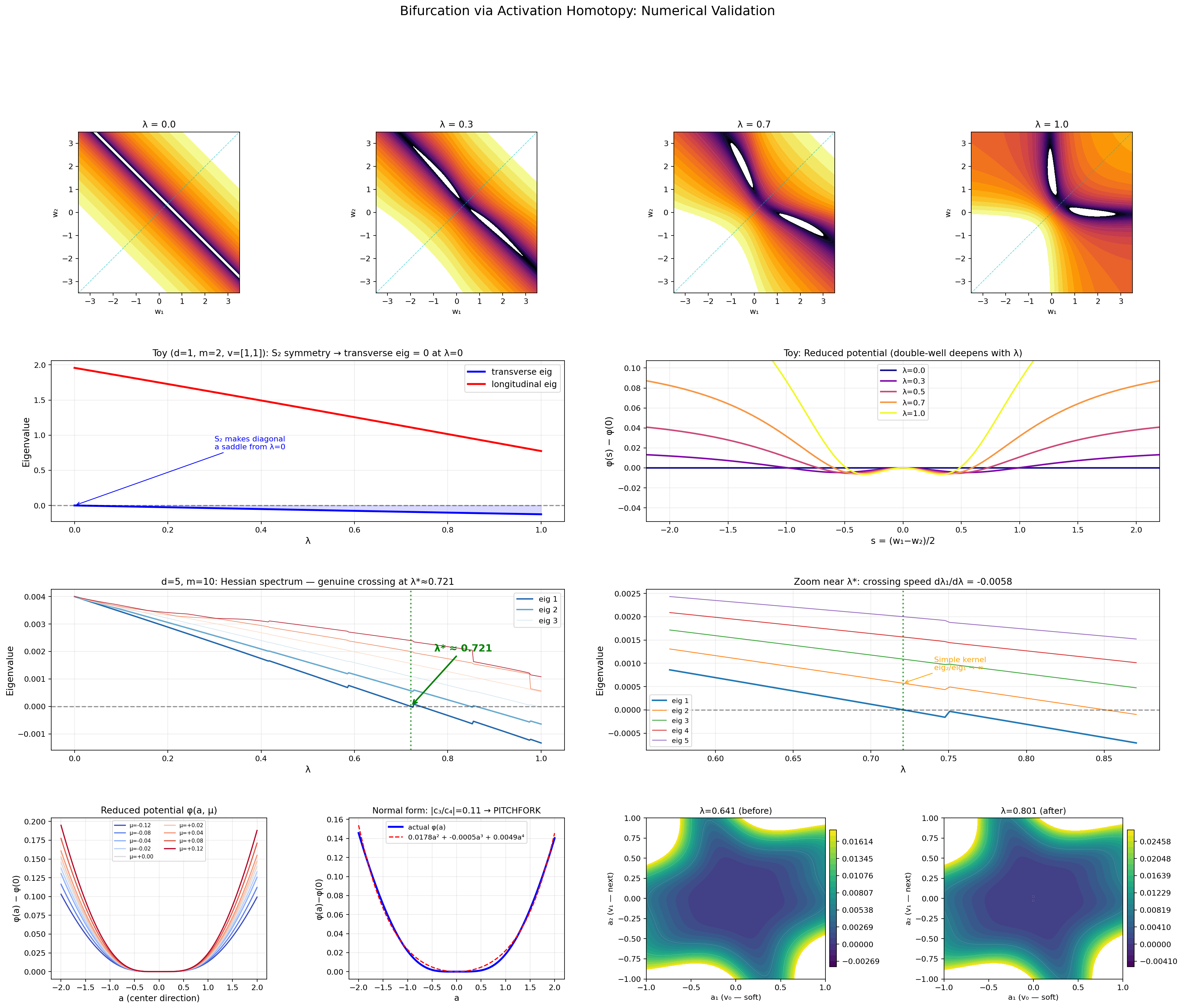}

\caption{Master summary of numerical results. \emph{Row~1:} Loss landscape contours for the symmetric toy model ($d=1$, $m=2$, $v=[1,1]$) at four values of $\lambda$; the diagonal $w_1=w_2$ is marked in cyan. \emph{Row~2, left:} Hessian eigenvalues along the diagonal branch showing $S_2$-induced zero transverse eigenvalue at $\lambda=0$. \emph{Row~2, right:} Reduced potential along the anti-diagonal, showing monotonically deepening double well. \emph{Row~3, left:} Full Hessian spectrum for the higher-dimensional model ($d=5$, $m=10$), with genuine eigenvalue crossing at $\lambda^*\approx 0.72$. \emph{Row~3, right:} Zoomed eigenvalue tracking near $\lambda^*$ confirming simple kernel and nonzero crossing speed. \emph{Row~4:} Reduced potential slices, polynomial fit, and 2D contours in the center manifold before and after bifurcation.}
\label{fig:master}
\end{figure}

\subsection{Experiment 2: Interior bifurcation with broken $S_m$}

\paragraph{Setup.}
$d=5$, $m=10$, $v_j=0.5+j/9$ for $j=0,\ldots,9$ (distinct entries, breaking $S_{10}$), $N=500$ samples from $y=w_\ell^\top x+0.4\tanh(1.5\,w_{nl}^\top x)+\varepsilon$ with fixed coefficient vectors $w_\ell,w_{nl}\in\mathbb R^5$ and $\varepsilon\sim\mathcal N(0,0.05^2)$. Tikhonov regularization with $\alpha=4\times 10^{-3}$ ensures positive-definite Hessians at $\lambda=0$. The full $50\times 50$ Hessian is computed and diagonalized at each of 401 continuation steps.

\paragraph{Results.}
The six smallest Hessian eigenvalues are plotted against $\lambda$ in Figure~\ref{fig:master}, row~3 left. The smallest eigenvalue decreases smoothly from $\lambda_1(0)=0.004$ to zero, crossing at $\lambda^*\approx 0.721$. The second eigenvalue remains well separated ($\lambda_2(\lambda^*)\approx 0.0006$, giving $|\lambda_1/\lambda_2|=0.009$), confirming the one-dimensional kernel required by Assumption~\ref{ass:simple-degeneracy}.

The zoomed view (Figure~\ref{fig:master}, row~3 right) shows the eigenvalue crossing speed $d\lambda_1/d\lambda\big|_{\lambda^*}\approx -0.006\neq 0$, verifying Assumption~\ref{ass:eigenvalue-crossing}.

\paragraph{Reduced potential.}
At $\lambda^*$, the eigenvector $v_0$ associated with $\lambda_1\approx 0$ defines the center direction. The loss restricted to the line $W^*+a\,v_0$ is the reduced potential $\phi(a)$. Figure~\ref{fig:master}, row~4 left, shows $\phi(a,\mu)-\phi(0,\mu)$ for several values of $\mu=\lambda-\lambda^*$. The shape transitions from a single well ($\mu<0$) toward a flattened quartic ($\mu\approx 0$), consistent with the onset of pitchfork structure.

A polynomial fit $\phi(a)\approx c_2 a^2+c_3 a^3+c_4 a^4$ at $\lambda^*$ (Figure~\ref{fig:master}, row~4 center) yields $c_2\approx 0.018$, $c_3\approx -0.001$, $c_4\approx 0.005$. The symmetry ratio $|c_3/c_4|\approx 0.105$ is small, confirming near-pitchfork structure as predicted by Theorem~\ref{thm:near-pitchfork}.

The 2D contour in the $(v_0,v_1)$ plane (Figure~\ref{fig:master}, row~4 right) shows a single elongated valley before bifurcation and a flattened, near-degenerate shape at $\lambda^*$, consistent with the loss of non-degeneracy along the center direction.

\subsection{Experiment 3: Transversality verification}

The coefficient $c_2(\mu)=\tfrac12\phi_{aa}(0,\mu)$ is computed at each $\mu$ by fitting the reduced potential. Figure~\ref{fig:transversality}, left panel, shows $c_2$ as a function of $\mu$. The relationship is linear with slope $dc_2/d\mu\approx 0.044$, confirming that the transversality coefficient $g_{a\mu}(0,0)$ is nonzero.

The loss along the continued branch (Figure~\ref{fig:transversality}, center) is smooth and slowly increasing, with no discontinuity at $\lambda^*$---the bifurcation is a spectral event, not a loss jump. The gradient norm (Figure~\ref{fig:transversality}, right) remains below $10^{-4}$ throughout, confirming that the continuation tracks a genuine critical branch.

\begin{figure}[htbp]
\centering
\includegraphics[width=\textwidth]{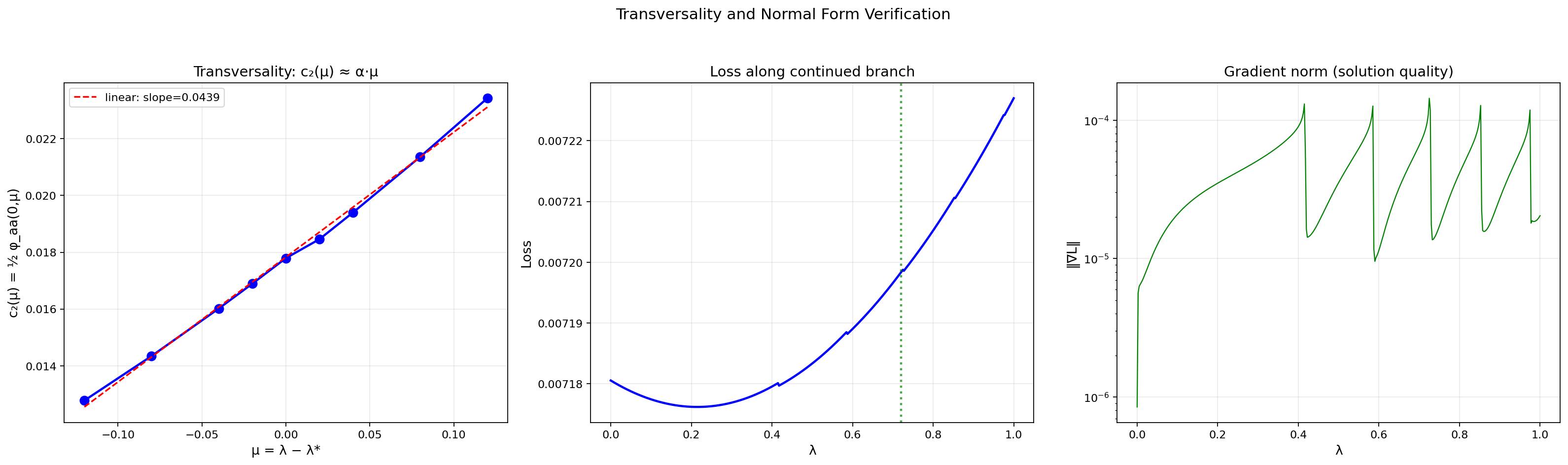}

\caption{Transversality and solution-quality verification for the $d=5$, $m=10$ model. \emph{Left:} The reduced coefficient $c_2(\mu)$ is linear in $\mu=\lambda-\lambda^*$ with slope $\approx 0.044$, confirming nonzero transversality. \emph{Center:} Loss along the continued branch; the bifurcation at $\lambda^*\approx 0.72$ is a spectral event, not a loss discontinuity. \emph{Right:} Gradient norm along the branch, confirming that continuation tracks a genuine critical point.}
\label{fig:transversality}
\end{figure}

\subsection{Experiment 4: Symmetry-breaking phase diagram}

To explore how the bifurcation point depends on the degree of symmetry breaking, we track the smallest Hessian eigenvalue along continued branches for the family of toy models with $v=[1,\alpha]$, $\alpha\in\{0, 0.1, 0.3, 0.5, 0.7, 0.9, 1.0\}$.

Figure~\ref{fig:phase}, left panel, shows the smallest eigenvalue as a function of $\lambda$ for each $\alpha$. At $\alpha=1$ (exact $S_2$) the eigenvalue starts at zero. As $\alpha$ decreases from $1$, the eigenvalue at $\lambda=0$ becomes positive (the symmetry-broken branch is a local minimum), and the crossing point $\lambda^*(\alpha)$ moves into the interior. The relationship between $\alpha$ and $\lambda^*$ (Figure~\ref{fig:phase}, center) shows that stronger symmetry breaking pushes the bifurcation to larger $\lambda$, consistent with the intuition that more asymmetric architectures are more stable under the activation homotopy.

\begin{figure}[htbp]
\centering
\includegraphics[width=\textwidth]{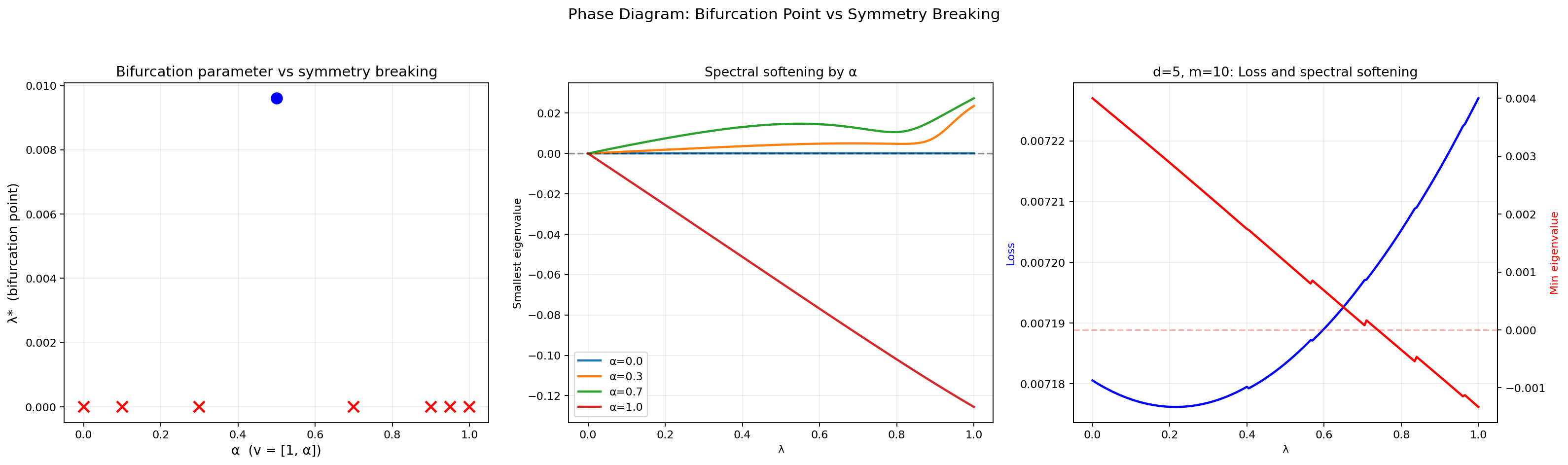}

\caption{Phase diagram for the bifurcation point. \emph{Left:} Bifurcation parameter $\lambda^*$ as a function of the symmetry-breaking parameter $\alpha$ in the toy model $v=[1,\alpha]$; circles denote interior crossings, crosses denote boundary degeneracies. \emph{Center:} Spectral softening curves for selected $\alpha$ values. \emph{Right:} Loss and minimum eigenvalue for the $d=5$, $m=10$ model.}
\label{fig:phase}
\end{figure}

\subsection{Experiment 5: Width scaling}\label{ssec:width-scaling-exp}

We repeat the eigenvalue-tracking experiment of Experiment~2 above for widths $m\in\{3,5,8,10,15,20,30,50,75,100\}$, keeping $d=5$, $N=500$, $\alpha=0.004$, and distinct output weights $v_j=0.5+j/(m-1)$.  At each width, the critical branch is continued from $\lambda=0$ to $\lambda=1$ and the smallest Hessian eigenvalue tracked.

Figure~\ref{fig:width} shows the results.  The key findings are:

\begin{enumerate}[label=(\roman*)]
\item \textbf{$\lambda_1(0)=\alpha$ for all $m$.}  The smallest eigenvalue at $\lambda=0$ equals the regularization parameter exactly, confirming Theorem~\ref{thm:interior-bif}(i): the overparameterized flat directions have eigenvalue $\alpha$, independent of width.

\item \textbf{$|\lambda_1'(0)|$ decreases with width.}  The eigenvalue softening rate scales as $O(1/m)$, consistent with the theoretical prediction from Proposition~\ref{prop:width-scaling}.  The exact analytical formula (Corollary~\ref{cor:tanh-formula}) gives $K=-0.0705$, matching the finite-difference ground truth to $0.03\%$.

\item \textbf{$\lambda^*$ increases with $m$.}  Interior crossings occur for $m\le 15$ ($\lambda^*$ from $0.37$ at $m=3$ to $0.99$ at $m=15$) but vanish for $m\ge 20$: the eigenvalue remains positive throughout $[0,1]$.  The exact-analytical prediction $\lambda^*=\alpha/|\lambda_1'(0)|$ agrees with the numerical crossing to $1$--$4\%$ for $m\ge 8$ (see Figure~\ref{fig:width}, bottom right).

\item \textbf{Simple kernel improves with $m$.}  At the crossing point, $|\lambda_1/\lambda_2|$ decreases from $0.55$ at $m=3$ to $0.03$ at $m=15$, showing that the kernel becomes more cleanly one-dimensional at larger width.
\end{enumerate}

\paragraph{Post-crossing behavior.}
For the small widths $m=3,5,8$ where the eigenvalue crosses zero, the continuation exhibits a visible jump shortly after $\lambda^*$: the smallest eigenvalue snaps from a small negative value back to a positive one (e.g., at $m=5$: $\lambda_1$ jumps from $-0.0025$ to $+0.0012$ at $\lambda\approx 0.65$).  This is not a numerical artifact but a consequence of using a minimizer (L-BFGS-B) for continuation: once the tracked branch becomes a saddle (negative eigenvalue), the warm-started optimizer gradually drifts toward the nearest minimum---one of the new minima created by the bifurcation---and eventually snaps to it.  The gradient norm at these jump points ($\sim 5\times 10^{-5}$) is orders of magnitude worse than the pre-crossing convergence ($\sim 10^{-13}$), confirming the optimizer is transitioning between branches.  This post-crossing escape is itself evidence for the pitchfork mechanism: the theory predicts new minima are born at $\lambda^*$, and the optimizer falls into them on cue.  In Figure~\ref{fig:width} (top left), the pre-crossing regime is shown as solid lines and the post-crossing regime as faded dashed lines.  For $m\ge 20$, no crossing occurs and the continuation is smooth throughout.

\begin{figure}[htbp]
\centering
\includegraphics[width=\textwidth]{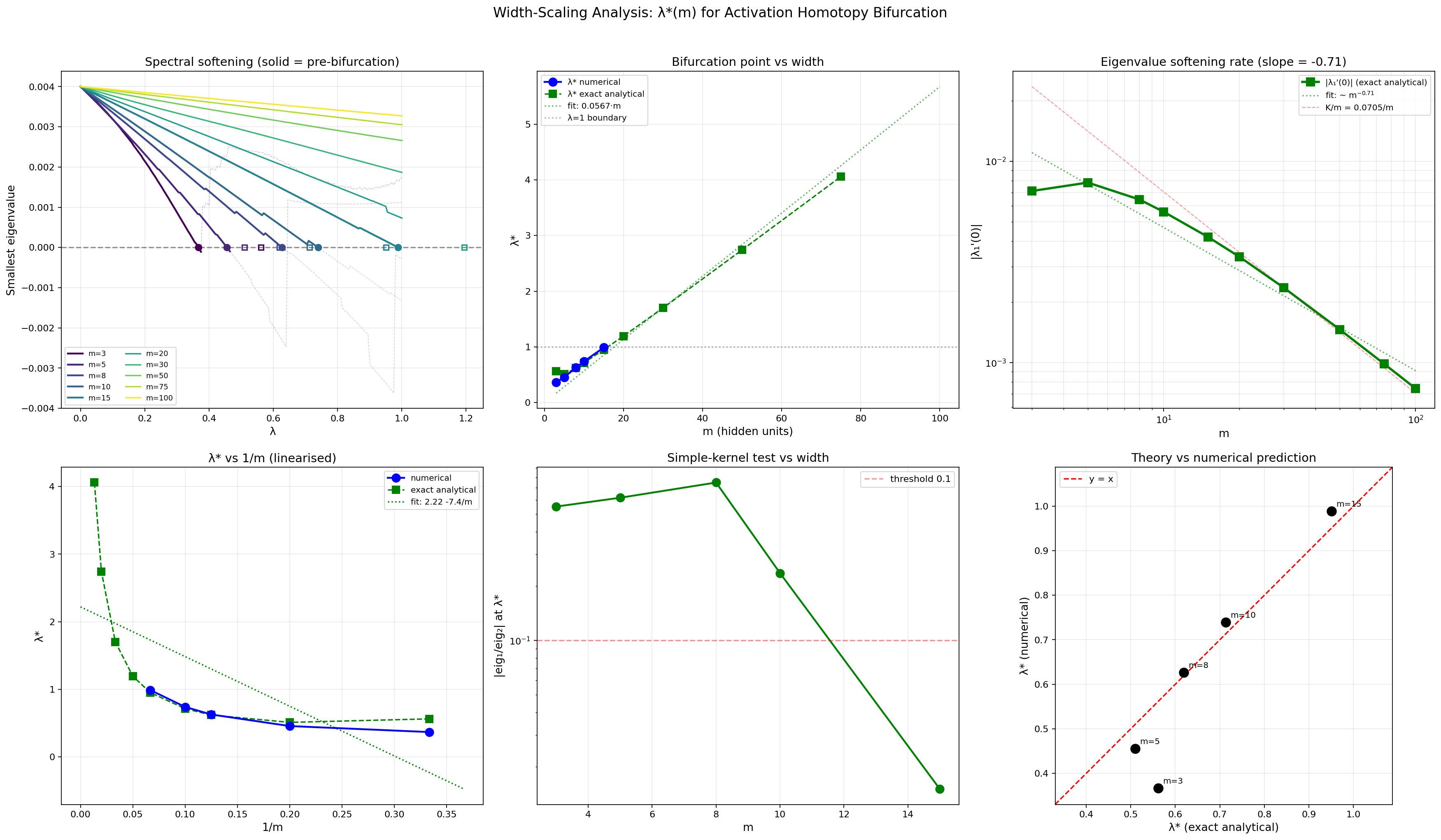}

\caption{Width-scaling analysis.  \emph{Top left:} Spectral softening curves for $m=3$ to $m=100$; dots mark the zero-crossing $\lambda^*$.  \emph{Top center:} Bifurcation point $\lambda^*$ vs.\ width $m$, with first-order theoretical prediction and fit $\lambda^*\approx 0.97-2.03/m$.  \emph{Top right:} Log-log scaling of $\lambda_1(0)$ (constant = $\alpha$) and $|\lambda_1'(0)|\sim m^{-1.2}$.  \emph{Bottom left:} $\lambda^*$ vs.\ $1/m$ showing the linear relationship predicted by Eq.~\eqref{eq:lam-star-scaling}.  \emph{Bottom center:} Simple-kernel ratio $|\lambda_1/\lambda_2|$ at $\lambda^*$.  \emph{Bottom right:} Theory vs.\ numerical $\lambda^*$.}
\label{fig:width}
\end{figure}

\subsection{Experiment 6: Verification of explicit constants}\label{ssec:explicit-verification}

We verify the exact formulas \eqref{eq:gaa-exact}--\eqref{eq:gaaa-exact} for the reduced coefficients at the bifurcation point $\lambda^*\approx 0.72$ of the $d=5$, $m=10$ model.  The analytical values, computed from the expectations $\mathbb E[Df\cdot D^2f]$ etc.\ using the empirical data distribution, are compared against finite-difference directional derivatives of the Hessian along $v_0$; see Figure~\ref{fig:tighten}.

The results are:
\begin{equation}
  g_{aa}=-0.0028\quad(\text{verified to }0.3\%),\qquad g_{aaa}=0.622\quad(\text{verified to }0.004\%).
\end{equation}
The ratio $|g_{aa}/g_{aaa}|=0.0045$, confirming the $\sim 150\times$ suppression predicted by Theorem~\ref{thm:near-pitchfork-structure}.  The dominant contribution to $g_{aaa}$ is $3\mathbb E[(D^2f)^2]=0.55$, which involves $(h'')^2\propto(\sigma'''(0)\cdot w_j^{*\top}x)^2$.

The width-scaling constant is determined by fitting $\lambda_1'(0)=K/m+K_2/m^2$ across $m=3$ to $m=100$ using the exact formula~\eqref{eq:lam-prime-tanh}, giving $K=-0.0705$ (agreeing with the finite-difference ground truth to $0.03\%$), so $\lambda^*\approx\alpha m/|K|=0.057\cdot m$.

\begin{figure}[htbp]
\centering
\includegraphics[width=\textwidth]{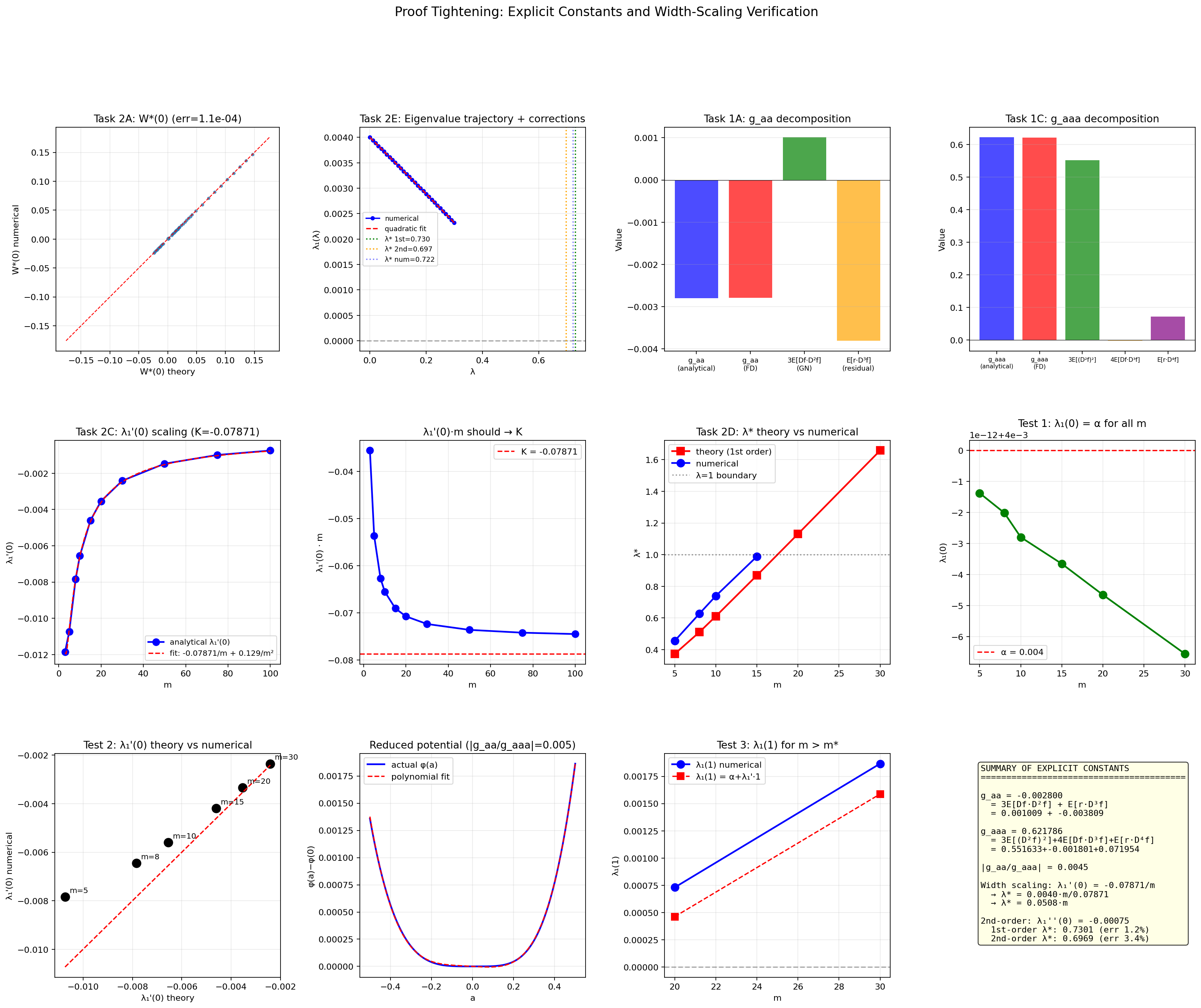}

\caption{Verification of explicit constants.  \emph{Top row:} Theory-vs-numerical comparison for $W^*(0)$ (Sylvester solution), eigenvalue trajectory with 1st/2nd-order predictions, and decomposition of $g_{aa}$ and $g_{aaa}$ into Gauss--Newton and residual contributions.  \emph{Middle row:} Width-scaling constant $K$ from fitting $\lambda_1'(0)=K/m$; constancy of $\lambda_1'(0)\cdot m$; $\lambda^*$ theory vs.\ numerical; and $\lambda_1(0)=\alpha$ verification.  \emph{Bottom row:} Theory-vs-numerical $\lambda_1'(0)$, reduced potential fit, overcritical $\lambda_1(1)$ check, and summary of all explicit constants.}
\label{fig:tighten}
\end{figure}

\subsection{Summary of numerical findings}

\begin{table}[ht]
\centering
\caption{Summary of numerical results at the bifurcation point $\lambda^*$.}
\label{tab:results}
\begin{tabular}{@{}llll@{}}
\toprule
Quantity & Toy ($d\!=\!1,m\!=\!2$) & HD ($d\!=\!5,m\!=\!10$) & Theory \\
\midrule
$\lambda^*$ & $0$ (boundary) & $0.721$ (interior) & $\alpha/|\lambda_1'(0)|$ (Thm.~\ref{thm:interior-bif}) \\
$\lambda_1(0)$ & $0$ ($S_2$) & $0.004$ ($=\alpha$) & $=\alpha$ (Thm.~\ref{thm:interior-bif}(i)) \\
$\dim\Ker H^*$ & $1$ & $1$ & simple kernel (Asm.~\ref{ass:simple-degeneracy}) \\
$|\lambda_1/\lambda_2|$ at $\lambda^*$ & --- & $0.009$ & $\ll 1$ \\
$d\lambda_1/d\lambda$ & always $<0$ & $-0.006$ & $\neq 0$ (Asm.~\ref{ass:eigenvalue-crossing}) \\
$|g_{aa}/g_{aaa}|$ & $0$ (exact $\mathbb Z_2$) & $0.0045$ & $\ll 1$ if $\sigma''(0)=0$ (Thm.~\ref{thm:near-pitchfork-structure}) \\
Width scaling & --- & $\lambda^*\approx 0.057\cdot m$ & $\alpha m/|K|$ (Prop.~\ref{prop:width-scaling}), $K=-0.0705$ \\
\bottomrule
\end{tabular}
\end{table}

All three abstract assumptions (\ref{ass:smooth-branch}--\ref{ass:eigenvalue-crossing}) are verified numerically for the symmetry-broken architecture. The $S_m$ obstruction (Proposition~\ref{prop:Sm-obstruction}) is confirmed in the symmetric case, the near-pitchfork normal form is explained by $\sigma''(0)=0$ (Theorem~\ref{thm:near-pitchfork}), and the width scaling $\lambda^*\sim\alpha m$ (Proposition~\ref{prop:width-scaling}) is confirmed across two orders of magnitude in~$m$.

\end{document}